\documentclass[10pt]{amsart}
\usepackage[margin=1in]{geometry}
\usepackage{amssymb}
\usepackage{mathrsfs}
\usepackage{mathabx,epsfig} 
\usepackage{cite}
\usepackage{amscd}
\usepackage{amsmath}
\usepackage{algorithmicx}
\usepackage{algpseudocode}
\usepackage{stmaryrd}
\usepackage{enumerate}
\usepackage{pict2e} 
\usepackage[all]{xy} 
\usepackage[usenames,dvipsnames]{color}

\usepackage[sc]{mathpazo}
\usepackage[T1]{fontenc}

\usepackage{hyperref}

\newtheorem{thm}{Theorem}
\newtheorem{pro}[thm]{Proposition}
\newtheorem{lmm}[thm]{Lemma}

{\theoremstyle{definition} 
\newtheorem{dfn}[thm]{Definition}
\newtheorem{ex}[thm]{Example}
\newtheorem{nonex}[thm]{Non-Example}
}

\newcommand{\N}{\mathbb{N}}
\newcommand{\Z}{\mathbb{Z}}
\newcommand{\A}{\mathbb{A}}
\newcommand{\V}{\mathbb{V}}
\newcommand{\CC}{\mathbb{C}}
\newcommand{\Spec}{{\rm Spec}\,}
\newcommand{\supp}{{\rm supp}}
\newcommand{\st}{\mathsf{st}}
\newcommand{\IN}{\mathsf{in}}
\newcommand{\es}{\mathsf{S}}
\newcommand{\et}{\text{\rm \'et}}
\newcommand{\punc}{{\rm punc}}
\newcommand{\lin}{{\rm lin}}
\newcommand{\lex}{{\rm lex}}
\newcommand\rows{{\rm rows}}
\newcommand\cols{{\rm cols}}
\newcommand{\cfoursecond}{\sideset{}{_2}\sum}
\newcommand{\cfourfirst}{\sideset{}{_1}\sum}
\newcommand\rank{{\rm rank}}
\def\acts{\mathrel{\reflectbox{$\righttoleftarrow$}}}

\newcommand\defn[1]{{\bf #1}}

\begin{document} 

\title{Combinatorial duality of Hilbert schemes of points in the affine plane}
\author{Mathias Lederer}
\email{mlederer@math.uni-bielefeld.de}
\address{Department of Mathematics \\ 
University of Innsbruck \\ 
Technikerstrasse 21a \\
 A-6020 Innsbruck \\ 
 Austria}
\thanks{The author was partially supported by a Marie Curie International Outgoing Fellowship 
of the EU Seventh Framework Program}
\date{\today}

\begin{abstract} 
  The Hilbert scheme of $n$ points in the affine plane contains the open subscheme parametrizing $n$ distinct points in the affine plane, 
  and the closed subscheme parametrizing ideals of codimension $n$ supported at the origin of the affine plane. 
  Both schemes admit Bia\l ynicki-Birula decompositions into moduli spaces of ideals with prescribed lexicographic Gr\"obner deformations. 
  We show that both decompositions are stratifications in the sense that the closure of each stratum is a union of more strata. 
  We show that the corresponding two partial orderings on the set of of monomial ideals are dual to each other. 
\end{abstract}

\maketitle


\section{Introduction}\phantomsection\label{sec:intro}

The polynomial ring $S := \CC[x_1,x_2]$ in two variables is the coordinate ring of the affine plane $\A^2$ over the complex numbers. 
The \defn{Hilbert scheme of points in the affine plane} $H^n(\A^2)$ is the moduli space representing the functor 
which sends each $\CC$-algebra $A$ to the set of ideals $I \subseteq S \otimes_\CC A$ such that the quotient 
$(S \otimes_\CC A) / I$ is a locally free $A$-module of dimension $n$. 
However, since we will only be working with $\CC$-valued points, 
we think of $H^n(\A^2)$ as the moduli space of ideals of codimension $n$ in $S$, 
\[
  H^n(\A^2) := \bigl\{ \text{ideals } I \subseteq S : \dim(S / I) = n \bigr\} . 
\]
The scheme $H^n(\A^2)$ is a smooth variety of dimension $2n$ \cite{Fogarty_smoothness, haimanCatalan}. 

The Hilbert scheme of points contains a subscheme 
\[
  H^{n,\punc}(\A^2) := \bigl\{ \text{ideals } I \subseteq S : \dim(S / I) = n, \, \supp(S / I) = \V(x_1,x_2) \bigr\} 
\]
called the \defn{punctual Hilbert scheme}. 
Upon identifying each ideal $I \subseteq S$ defining a point in $H^{n,\punc}(\A^2)$ with its affine subscheme $\Spec(S/I) \subset \A^2$, 
we think of each point in $H^{n,\punc}(\A^2)$ as a \defn{fat point} of fatness $n$ supported at at the origin. 
Each ideal $I$ defining a point in $H^n(\A^2)$ can be factored into $I = I_0 \cap \ldots \cap I_m$ 
such that $\supp(S / I_j) \neq \supp(S / I_k)$ for $j \neq k$, 
and each $I_j$ defines a point in a translated version of $H^{n_j,\punc}(\A^2)$, where $n_j = \dim(S / I_j)$. 
The punctual Hilbert scheme may therefore be viewed as the fundamental building block from which the Hilbert scheme of points is assembled. 
This is the ultimate reason why the punctual Hilbert scheme is of crucial importance many aspects of the theory of Hilbert schemes, including 
\begin{itemize}
  \item Ellingsrud and Str\o mme determination of the Betti numbers of the Hilbert scheme of points in the projective plane \cite{esBetti, esCells};
  \item Grojnowski's and Nakajima's implementation of the ring of symmetric functions in terms of Hilbert schemes of points on surfaces
  \cite{grojnowski, nakajimaAnnals, nakajima};
  \item the ongoing project of describing the non-smooth, non-reduced and non-irreducible nature of the Hilbert schemes of points 
  in $d$-space for $d>2$, the contributors to which include Iarrobino, Emsalem \cite{iarrobino72, iarrobinoEmsalem}, 
  Mazzola \cite{mazzola}, Cartwright, Erman, Velasco and Viray \cite{velasco, ermanMurphy} and many others. 
\end{itemize}

The cited papers by Ellingsrud and Str\o mme are built around three actions of the one-dimensional torus $\CC^\star$ on $S$ inducing 
Bia\l ynicki-Birula decompositions of three types of Hilbert schemes. 
All actions are given by weights $v = (v_1,v_2) \in \Z$. 
Upon using the multi-index notation $x^\alpha := x_1^{\alpha_1}x_2^{\alpha_2}$, the actions are defined by 
\[
  \mathbb{G}_m \acts S : 
  t.x^\alpha = t^{\langle \alpha, v \rangle}x^\alpha , 
\]
where $\langle \alpha, v \rangle = \alpha_1v_1 + \alpha_2v_2$. The first action is defined by a weight vector such that $v_1 \ll v_2 < 0$, 
more precisely, $v_1 \leq nv_2 < 0$. 
The fixed points of this action are monomial ideals $M_\Delta \subseteq S$ of codimension $n$. 
The indexing set of $M_\Delta$ is the corresponding \defn{standard set}, or \defn{staircase} $\Delta \subseteq \N^2$ of cardinality $n$. 
In general, a standard set is a subset of $\N^2$ whose complement is closed under addition of elements from $\N^2$; 
the standard set of a monomial ideal contains all exponents of monomials not showing up in the ideal. 
The Bia\l ynicki-Birula cells of the action (containing all ideals $I \subseteq S$ floating into $M_\Delta$ as $t \to 0$) 
are \defn{lexicographic Gr\"obner strata} or \defn{lexicographic Gr\"obner basins}
\[
  H^\Delta_\lex(\A^2) := \bigl\{ \text{ideals } I \subseteq S : \IN_\lex(I) = M_\Delta \bigr\} .
\]
Here $\IN_\lex(I)$ denotes the lexicographic initial ideal, or lexicographic Gr\"obner deformation of $I$, 
with respect to the lexicographic ordering in which $x_1 > x_2$. 
Thus the Bia\l ynicki-Birula decomposition reads 
\begin{equation}
\label{bbA2}
  H^n(\A^2) = \coprod_{\Delta \in \st_n} H^\Delta_\lex(\A^2) ,
\end{equation}
where $\st_n$ denotes the set of staircases of cardinality $n$. 
This last equation just rephrases the elementary fact that every ideal admits a unique lexicographic Gr\"obner basis. 

The second type of action which Ellingsrud and Str\o mme used will be discussed later in the paper. 
The third type is given by a weight vector such that $0 < v_1 \ll v_2$, more precisely, $0 < nv_1 \leq v_2$. 
The fixed points of this action are also the standard sets $M_\Delta$ for $\Delta \in \st_n$. 
The limit $\lim_{t \to 0} t.I$ exists in $S$ only if $I$ is supported at the origin. 
This follows from the simple observation our torus action on $\A^2$ is given by $t.(a_1,a_2) = (t^{-v_1}a_1,t^{-v_2}a_2)$ on coordinates, 
which will fly off to infinity unless $(a_1,a_2) = (0,0)$. 
As for a schematic treatment of this observation, see \cite{just_bb}. 
It therefore only makes sense to consider this action on $H^{n,\punc}(\A^2)$. 
The Bia\l ynicki-Birula decomposition then reads 
\begin{equation}
\label{bbA2punc}
  H^{n,\punc}(\A^2) = \coprod_{\Delta \in \st_n} H^{\Delta,\punc}_\lex(\A^2) ,
\end{equation}
where 
\[
  H^{\Delta,\punc}_\lex(\A^2) := H^\Delta_\lex(\A^2) \cap H^{n,\punc}_\lex(\A^2) .
\]
Remember that $H^{n,\punc}(\A^2)$ parametrizes points of fatness $n$ supported at the origin. 
Equation \eqref{bbA2punc} thus distinguishes those points according to the shapes of their respective fatnesses. 

Decomposition \eqref{bbA2punc} arises from decomposition \eqref{bbA2} by intersecting both sides with the closed subscheme 
$H^{n,\punc}(\A^2)$ of $H^n(\A^2)$. 
Consider the open subscheme 
\[
  H^{n,\et}(\A^2) := \bigl\{ \text{ideals } I \subseteq S : \dim(S / I) = n, \, |\supp(S/I)| = n \bigr\} 
\]
of $H^n(\A^2)$. 
The superscript stands for \defn{\'etale}, 
since $H^{n,\et}(\A^2)$ represents the functor which sends each $\CC$-algebra $A$ 
to the set of ideals $I \subseteq S \otimes_\CC A$ such that the projection  
$\Spec (S \otimes_\CC A) / I \to \Spec A$ is an \'etale morphism of degree $n$ \cite{components}. 
The two schemes $H^{n,\punc}(\A^2)$ and $H^{n,\et}(\A^2)$ may be viewed as antagonists of each other, 
since the first parametrizes ideals whose corresponding scheme is a fat point supported at the origin, 
whereas the second parametrizes ideals whose corresponding scheme contains no fat points at all. 
Perhaps a more precise way of viewing $H^{n,\punc}(\A^2)$ and $H^{n,\et}(\A^2)$ as antipodes of each other is 
the fact that $H^{n,\punc}(\A^2)$ is the preimage of the $n$-fold origin $n \cdot 0$ under the \defn{Hilbert-Chow morphism}
\[
  \alpha_n : H^{n,\et}(\A^2) \to (\A^2)^{(n)} := (\A^2)^n / S_n ,
\] 
whereas $H^{n,\et}(\A^2)$ is the open locus where the Hilbert-Chow morphism is an isomorphism \cite[Section 2]{bertin}. 
Upon intersecting both sides of \eqref{bbA2} with the open subscheme $H^{n,\et}(\A^2)$ of $H^n(\A^2)$, 
we obtain 
\begin{equation}
\label{bbA2et}
  H^{n,\et}(\A^2) = \coprod_{\Delta \in \st_n} H^{\Delta,\et}_\lex(\A^2) .
\end{equation}
The goal of the present paper is to show that decompositions \eqref{bbA2punc} and \eqref{bbA2et} are antagonists of each other in a combinatorial sense. 
The crucial notion for this are two partial orderings on $\st_n$. 

\begin{dfn}
\label{dfn:twoOrderings}
  \begin{itemize}
    \item Take an arbitrary $\Delta \in \st_n$, 
    and let $R := \{ \Delta_0, \ldots, \Delta_l \}$ be the multiset of rows of $\Delta$. 
    So each $\Delta_i$ is a one-dimensional standard set, i.e., a finite interval in $\N$ starting at $0$. 
    Let $R = R_0 \coprod \ldots \coprod R_m$ be any partition of that multiset. 
    For each $j$, we let $\Delta'_j$ be the one-dimensional standard set 
    whose cardinality is the sum of the cardinalities of elements of $R_j$. 
    Then the multiset $R' := \{ \Delta'_0, \ldots, \Delta'_m \}$ is the multiset of rows of a standard set $\Delta' \in \st_n$. 
    \item We define $\Delta \leq_\et \Delta'$ if, and only if, $\Delta'$ arises from $\Delta$ by the above-described process. 
    \item We define $\Delta \leq_\punc \Delta'$ if, and only if, $(\Delta')^t \leq_\et \Delta^t$, 
    the superscript $t$ standing for transposition of standard sets. 
  \end{itemize}
\end{dfn}
In the terminology of \cite[p.103]{Macdonald}, $\Delta \leq_\et \Delta'$ if \defn{$\Delta$ is a refinement of $\Delta'$}. 
Figures \ref{fig:etalePartialOrdering} and \ref{fig:punctualPartialOrdering} show the Hasse diagrams of $\leq_\et$ and $\leq_\punc$, 
respectively, on standard sets of size 6. Arrows point from smaller to larger elements. 

\begin{center}
\begin{figure}[ht]
  \begin{picture}(430,130)
     \unitlength0.36mm
     \multiput(10,35)(10,0){2}{\line(0,1){60}}
     \multiput(10,35)(0,10){7}{\line(1,0){10}}
     \put(25,60){\vector(1,0){30}}
     \multiput(60,40)(10,0){2}{\line(0,1){50}}
     \put(80,40){\line(0,1){10}}
     \multiput(60,40)(0,10){2}{\line(1,0){20}}
     \multiput(60,60)(0,10){4}{\line(1,0){10}}
     \put(85,65){\vector(3,2){30}}
     \put(85,55){\vector(3,-2){30}}
     \multiput(120,70)(10,0){2}{\line(0,1){40}}
     \multiput(140,70)(10,0){2}{\line(0,1){10}}
     \multiput(120,70)(0,10){2}{\line(1,0){30}}
     \multiput(120,90)(0,10){3}{\line(1,0){10}}
     \put(155,85){\vector(2,1){30}}
     \put(155,75){\vector(2,-1){30}}
     \multiput(120,20)(10,0){2}{\line(0,1){40}}
     \put(140,20){\line(0,1){20}}
     \multiput(120,20)(0,10){3}{\line(1,0){20}}
     \multiput(120,50)(0,10){2}{\line(1,0){10}}
     \put(155,30){\vector(4,1){30}}
     \put(155,40){\vector(2,1){30}}
     \put(155,50){\vector(2,3){30}}
     \multiput(190,90)(10,0){2}{\line(0,1){30}}
     \multiput(190,90)(0,10){2}{\line(1,0){30}}
     \put(210,90){\line(0,1){20}}
     \put(220,90){\line(0,1){10}}
     \put(190,110){\line(1,0){20}}
     \put(190,120){\line(1,0){10}}
     \put(235,100){\vector(1,0){30}}
     \put(235,94){\vector(1,-1){30}}
     \put(235,88){\vector(1,-2){30}}
     \multiput(190,50)(10,0){2}{\line(0,1){30}}
     \multiput(190,70)(0,10){2}{\line(1,0){10}}
     \multiput(190,50)(0,10){2}{\line(1,0){40}}
     \multiput(210,50)(10,0){3}{\line(0,1){10}}
     \put(235,60){\vector(1,0){30}}
     \put(235,53){\vector(1,-1){30}}
     \multiput(190,10)(10,0){3}{\line(0,1){30}}
     \multiput(190,10)(0,10){4}{\line(1,0){20}}
     \put(235,20){\vector(1,0){30}}
     \multiput(270,90)(10,0){4}{\line(0,1){20}}
     \multiput(270,90)(0,10){3}{\line(1,0){30}}
     \put(325,90){\vector(1,-1){30}}
     \multiput(270,50)(10,0){2}{\line(0,1){20}}
     \put(270,70){\line(1,0){10}}
     \multiput(270,50)(0,10){2}{\line(1,0){50}}
     \multiput(290,50)(10,0){4}{\line(0,1){10}}
     \put(325,55){\vector(1,0){30}}
     \multiput(270,10)(10,0){3}{\line(0,1){20}}
     \multiput(270,10)(0,10){2}{\line(1,0){40}}
     \put(270,30){\line(1,0){20}}
     \multiput(300,10)(10,0){2}{\line(0,1){10}}
     \put(325,20){\vector(1,1){30}}
     \multiput(360,50)(0,10){2}{\line(1,0){60}}
     \multiput(360,50)(10,0){7}{\line(0,1){10}}
   \end{picture}
  \caption{The \'etale partial ordering on $\st_6$}
  \phantomsection\label{fig:etalePartialOrdering}
\end{figure}
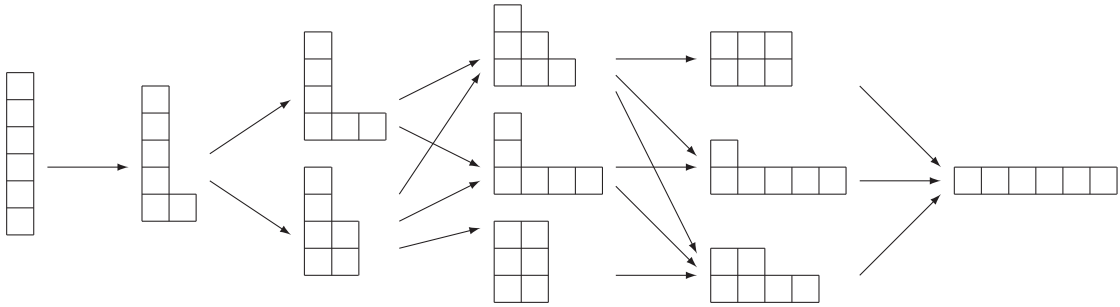
\end{center}

\begin{center}
\begin{figure}[ht]
  \begin{picture}(430,160)
     \unitlength0.36mm
     \multiput(10,50)(10,0){2}{\line(0,1){60}}
     \multiput(10,50)(0,10){7}{\line(1,0){10}}
     \put(25,85){\vector(2,3){30}}
     \put(25,80){\vector(1,0){30}}
     \put(25,75){\vector(2,-3){30}}
     \multiput(60,120)(10,0){3}{\line(0,1){30}}
     \multiput(60,120)(0,10){4}{\line(1,0){20}}
     \put(85,135){\vector(1,0){30}}
     \multiput(60,60)(10,0){2}{\line(0,1){50}}
     \put(80,60){\line(0,1){10}}
     \multiput(60,60)(0,10){2}{\line(1,0){20}}
     \multiput(60,80)(0,10){4}{\line(1,0){10}}
     \put(85,80){\vector(1,0){30}}
     \put(85,85){\vector(2,3){30}}
     \multiput(60,10)(10,0){2}{\line(0,1){40}}
     \put(80,10){\line(0,1){20}}
     \multiput(60,10)(0,10){3}{\line(1,0){20}}
     \multiput(60,40)(0,10){2}{\line(1,0){10}}
     \put(85,30){\vector(1,3){30}}
     \put(85,25){\vector(1,1){30}}
     \put(85,20){\vector(1,0){30}}
     \multiput(120,120)(10,0){2}{\line(0,1){30}}
     \put(140,120){\line(0,1){20}}
     \put(150,120){\line(0,1){10}}
     \multiput(120,120)(0,10){2}{\line(1,0){30}}
     \put(120,140){\line(1,0){20}}
     \put(120,150){\line(1,0){10}}
     \put(155,130){\vector(3,-1){30}}
     \put(155,120){\vector(2,-3){30}}
     \multiput(120,55)(10,0){2}{\line(0,1){40}}
     \multiput(140,55)(10,0){2}{\line(0,1){10}}
     \multiput(120,55)(0,10){2}{\line(1,0){30}}
     \multiput(120,75)(0,10){3}{\line(1,0){10}}
     \put(155,75){\vector(2,-1){30}}
     \put(155,85){\vector(1,1){30}}
     \multiput(120,10)(10,0){4}{\line(0,1){20}}
     \multiput(120,10)(0,10){3}{\line(1,0){30}}
     \put(155,20){\vector(1,1){30}}
     \multiput(190,100)(10,0){2}{\line(0,1){30}}
     \multiput(210,100)(10,0){3}{\line(0,1){10}}
     \multiput(190,100)(0,10){2}{\line(1,0){40}}
     \multiput(190,120)(0,10){2}{\line(1,0){10}}
     \put(235,100){\vector(2,-1){30}}
     \multiput(190,50)(10,0){3}{\line(0,1){20}}
     \multiput(220,50)(10,0){2}{\line(0,1){10}}
     \multiput(190,50)(0,10){2}{\line(1,0){40}}
     \put(190,70){\line(1,0){20}}
     \put(235,60){\vector(2,1){30}}
     \multiput(270,70)(0,10){2}{\line(1,0){50}}
     \put(270,90){\line(1,0){10}}
     \multiput(270,70)(10,0){2}{\line(0,1){20}}
     \multiput(290,70)(10,0){4}{\line(0,1){10}}
     \put(325,75){\vector(1,0){30}}
     \multiput(360,70)(0,10){2}{\line(1,0){60}}
     \multiput(360,70)(10,0){7}{\line(0,1){10}}
   \end{picture}
  \caption{The punctual partial order on $\st_6$}
  \phantomsection\label{fig:punctualPartialOrdering}
\end{figure}
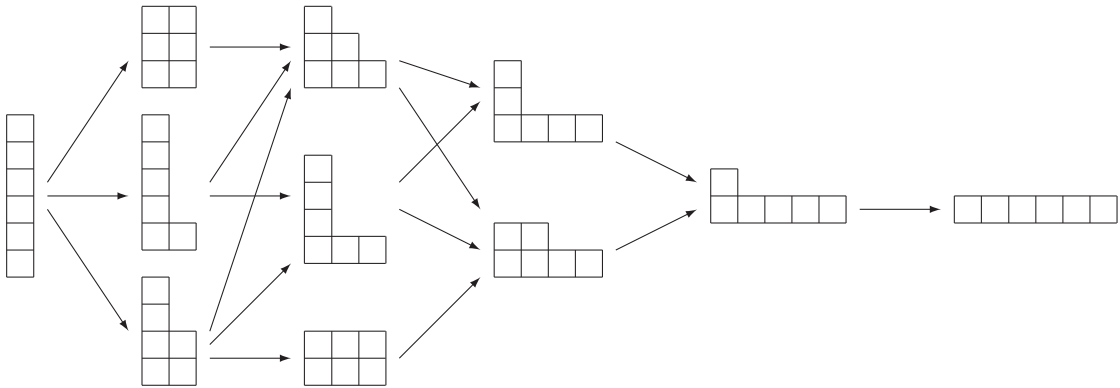
\end{center}

\begin{thm}
\phantomsection\label{thm:combinatorialDuality}
  \begin{enumerate}[(i)]
    \item The closure of $H^{\Delta,\et}_\lex(\A^2)$ in $H^{n,\et}(\A^2)$ is the union of all $H^{\Delta',\et}_\lex(\A^2)$ 
    such that $\Delta \leq_\et \Delta'$. 
    \item The closure of $H^{\Delta,\punc}_\lex(\A^2)$ in $H^{n,\punc}(\A^2)$ is the union of all $H^{\Delta',\punc}_\lex(\A^2)$ 
    such that $\Delta \leq_\punc \Delta'$. 
    \item In particular, decompositions \eqref{bbA2punc} and \ref{bbA2et}
    are stratifications whose partial orderings on $\st_n$ are $\leq_\et$ and $\leq_\punc$, respectively. 
  \end{enumerate}
\end{thm}

Here we use the term \defn{stratification} in the strict sense, namely, 
for a decomposition 
\begin{equation}
\label{eqn:generalStratification}
  X = \coprod_{i \in \es} X_i
\end{equation}
of a topological space $X$ into locally closed subspaces $X_i$ such that the closure of each $X_i$ is a union of more $X_j$. 
This datum induces a partial ordering $\leq$ on the indexing set $\es$ in which $i \leq j$ if, and only if, $X_j \subseteq \overline{X_i}$. 

\subsection*{Outline of the paper}

In Section \ref{sec:c4twoDirections} we provide some basics about a combinatorial operation on standard sets called \defn{Connect Four} 
or \defn{C4} for short. 
This operation secretly defines the partial orderings $\leq_\et$ and $\leq_\punc$
and controls the combinatorics of lexicographic Gr\"obner basins in $H^n(\A^2)$.
The same operation also controls the geometry of lexicographic Gr\"obner basins in $H^{n,\et}(\A^d)$ when $d\geq3$ \cite{components}. 
In the plane case, statements analogous to those of the cited paper reduce to trivialities; 
the author's long-term project is to generalize the findings from the present paper from the plane case to $d\geq3$.
With the prerequisites from Section \ref{sec:c4twoDirections} at hand, the proof of Theorem \ref{thm:combinatorialDuality} (i) is not hard; 
we shall present it in Section \ref{sec:etaleIncidences}. 
The proof of Theorem \ref{thm:combinatorialDuality} (ii), which takes considerably more effort, 
shall be carried out in Section \ref{sec:punctualIncidences}. 
We conclude the paper with a few results and one conjecture 
on incidences among Gr\"obner basins in $H^n(\A^2)$ in Section \ref{sec:generalIncidences}. 

\section*{Acknowledgements}

I wish to thank Allen Knutson and Jenna Rajchgot for our longstanding research collaboration on Hilbert schemes---and many other things. 
Many thanks go to Anthony Iarrobino for inviting me to Boston, 
for giving me the opportunity to present my work at the GASC seminar and the AMS Special Session on Hilbert schemes, 
and for many fruitful and supportive comments about my work. 
I am particularly indebted to Laurent Evain. 
Laurent and I stated and proved Propositions \ref{pro:C4sumFirstDirection} and \ref{pro:C4sumSecondDirection} in joint effort. 
Special thanks go to my students Daniel Heinrich and Patrick Wegener. 
I wrote large parts of this paper while teaching a course on Hilbert schemes, 
in which I got much inspiration from the two. 
I gave a preliminary account of the work presented here at Max Planck Institut f\"ur Mathematik at Bonn; 
I wish to thank Bernd Sturmfels for inviting me to Bonn. 
I presented a finished version of my work at Universit\`a degli Studi di Genova; 
I wish to thank Aldo Conca for inviting me to Genova. 


\section{C4 in two different directions}
\label{sec:c4twoDirections}

For two-dimensional standard sets $\Delta$ and $\Delta'$, we define their \defn{C4 sum in the first direction} as 
\[
  \Delta +_1 \Delta' := \left\lbrace
  \alpha \in \N^2: 
  p_1(\alpha) < \left| p_2^{-1}\bigl( p_2(\alpha) \bigr) \cap \Delta \right|
  + \left| p_2^{-1}\bigl( p_2(\alpha) \bigr) \cap \Delta' \right|
  \right\rbrace 
\]
and their \defn{C4 sum in the second direction} as 
\[
  \Delta +_2 \Delta' := \left\lbrace
  \alpha \in \N^2: 
  p_2(\alpha) < \left| p_1^{-1}\bigl( p_1(\alpha) \bigr) \cap \Delta \right|
  + \left| p_1^{-1}\bigl( p_1(\alpha) \bigr) \cap \Delta' \right|
  \right\rbrace .
\]
So $\Delta +_1 \Delta'$ arises by arranging the columns of $\Delta$ and of $\Delta'$ in decreasing order from left to right, 
and $\Delta +_2 \Delta'$ arises by arranging the rows of $\Delta$ and of $\Delta'$ in decreasing order from bottom to top. 
Figure \ref{fig:c4bothDirections} shows C4 sums in either direction; 
the picture explains the terminology, reminiscent of the popular two-player game \emph{Connect Four}. 

\begin{center}
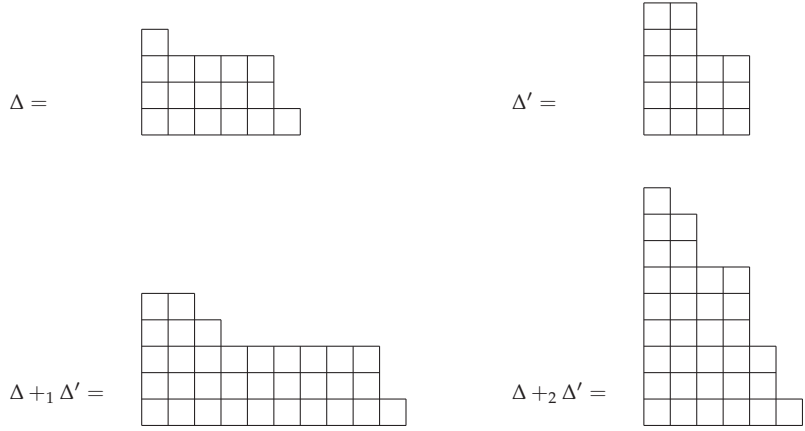
\begin{figure}[ht]
  \begin{picture}(320,170)
    \put(10,130){\footnotesize $\Delta =$}
    \multiput(60,120)(0,10){2}{\line(1,0){60}}
    \multiput(60,140)(0,10){2}{\line(1,0){50}}
    \put(60,160){\line(1,0){10}}
    \multiput(60,120)(10,0){2}{\line(0,1){40}}
    \multiput(80,120)(10,0){4}{\line(0,1){30}}
    \put(120,120){\line(0,1){10}}
    \put(200,130){\footnotesize $\Delta' =$}
    \multiput(250,120)(0,10){4}{\line(1,0){40}}
    \multiput(250,160)(0,10){2}{\line(1,0){20}}
    \multiput(250,120)(10,0){3}{\line(0,1){50}}
    \multiput(280,120)(10,0){2}{\line(0,1){30}}
    \put(10,20){\footnotesize $\Delta +_1 \Delta' =$}
    \multiput(60,10)(0,10){2}{\line(1,0){100}}
    \multiput(60,30)(0,10){2}{\line(1,0){90}}
    \put(60,50){\line(1,0){30}}
    \put(60,60){\line(1,0){20}}
    \multiput(60,10)(10,0){3}{\line(0,1){50}}
    \put(90,10){\line(0,1){40}}
    \multiput(100,10)(10,0){6}{\line(0,1){30}}
    \put(160,10){\line(0,1){10}}
    \put(200,20){\footnotesize $\Delta +_2 \Delta' =$}
    \multiput(250,10)(0,10){2}{\line(1,0){60}}
    \multiput(250,30)(0,10){2}{\line(1,0){50}}
    \multiput(250,50)(0,10){3}{\line(1,0){40}}
    \multiput(250,80)(0,10){2}{\line(1,0){20}}
    \put(250,100){\line(1,0){10}}
    \multiput(250,10)(10,0){2}{\line(0,1){90}}
    \put(270,10){\line(0,1){80}}
    \multiput(280,10)(10,0){2}{\line(0,1){60}}
    \put(300,10){\line(0,1){30}}
    \put(310,10){\line(0,1){10}}
  \end{picture}
  \caption{C4 addition in the first and in the second direction}
\phantomsection\label{fig:c4bothDirections}
\end{figure}
\end{center}

\begin{pro}
\cite[Section 7.2]{nakajima}. 
\phantomsection\label{pro:C4sumFirstDirection}
  For $i = 0, \ldots, m$, let $I_j \subseteq S$ be ideals in $H^{\Delta_j}_\lex(\A^2)$ with distinct supports on the axis $\V(x_2)$. 
  Let 
  \[
    I := \bigcap_{j = 0}^m I_j . 
  \]
  Then $\IN_\lex(I) = M_\Delta$, where 
  \[
    \Delta := \cfourfirst_{j = 0}^m \Delta_j. 
  \]
\end{pro}

\begin{proof}
  This statement appears without proof in the cited book, so let's briefly verify it here. 
  It suffices to prove the proposition in the case where each $I_j$ is supported in one point on the $x_1$-axis. 
  We call the elements of the minimal generating system of the $\N^2$-module $\N^2 \setminus \Delta$ the \defn{outer corners} of $\Delta$. 
  We have to prove, for each outer corner $\alpha$ of $\Delta$, 
  the existence of an $f \in I$ with initial term $x^\alpha$. 
  Upon defining 
  \[
    \alpha_{1,j} := \left| p_2^{-1}\bigl( p_2(\alpha) \bigr) \cap \Delta_j \right| ,
  \]
  we see that $(\alpha_{1,j},\alpha_2) \in \N^2 \setminus \Delta_j$, 
  which implies the existence of a polynomial
  $f_j \in I_j$ with initial exponent $(\alpha_{1,j},\alpha_2)$. 
  Lemma \ref{lmm:divisibility} implies that among those $f_j$, there is one divisible by $x_2^{\alpha_2}$, so $f_j = x_2^{\alpha_2} p_j$ 
  for some $p_j \in S$ with initial exponent $\alpha_{1,j}$: 
  If $(\alpha_{1,j},\alpha_2)$ is an outer corner of $\Delta_j$, then take the corresponding element of the lexicographic Gr\"obner basis.
  Otherwise, some $(\alpha_{1,j},\alpha_2 - b)$ is an outer corner of $\Delta_j$. 
  In this case, take $x_2^b$ times the corresponding element of the lexicographic Gr\"obner basis. 
  The polynomial
  \[
    f := x_2^{\alpha_2} \prod_{j = 0}^m p_j .
  \]
  then has the initial term $x^\alpha$ since $\alpha_1 = \alpha_{1,0} + \ldots + \alpha_{1,m}$ by definition of $\alpha_{1,j}$. 
\end{proof}

\begin{lmm}
\phantomsection\label{lmm:divisibility}
  Let $I \in H^\Delta_\lex(\A^2)$ be supported on $\V(x_2)$ and $\alpha$ be an outer corner of $\Delta$. 
  If $f_\alpha$ is the element of the reduced lexicographic Gr\"obner basis of $I$ with initial exponent $\alpha$ and trailing exponents in $\Delta$, 
  then $f_\alpha$ is divisible by $x_2^{\alpha_2}$. 
\end{lmm}

\begin{proof}
  The elements of the reduced lexicographic Gr\"obner basis of $I$ take the shape 
  \[
    f_\alpha = x^\alpha + \sum_{\beta \in \Delta, \beta < \alpha} c^\alpha_\beta x^\beta ,
  \] 
  where $\alpha$ runs through the set of outer corners of $\Delta$. 
  Consider the $\CC^\star$-action with a weight $v' \in \Z^2$ such that $v'_1 < 0$ and $v'_2 = 0$. 
  The ideal $I' := \lim_{t \to 0} t.I$ defines a point in $H^{\Delta,\punc}_\lex(\A^2)$. 
  The elements $f'_\alpha$ of its reduced lexicographic Gr\"obner basis of $I'$ 
  arise from the elements 
  of the reduced lexicographic Gr\"obner basis of $I$ by killing all coefficients $c^\alpha_\beta$ such that $\beta_1 < \alpha_1$. 
  Each polynomial $f'_\alpha$ therefore factors as $f'_\alpha = x_1^{\alpha_1} g'_\alpha$ for some $g'_\alpha \in \CC[x_2]$. 
  If $g'_\alpha$ was not divisible by $x_2^{\alpha_2}$, then it would contain a factor $x_2 - \zeta$ for some $\zeta \neq 0$. 
  Since the ideal $I'$ is supported at the origin, a factor $x_2 - \zeta$ is not allowed to show up in the reduced lexicographic Gr\"obner basis. 
  Therefore, $c^\alpha_\beta = 0$ for $\alpha_1 = \beta_1$ and $\beta_2 < \alpha_2$. 
  
  Let $P$ be the set of pairs $(\alpha,\beta)$ consisting of an outer corner $\alpha$ and an element $\beta$ of $\Delta$ 
  such that $c^\alpha_\beta \neq 0$ and $\beta_2 < \alpha_2$. 
  We order the elements of $P$ in the decreasing way according to the slope of the line through $\alpha$ and $\beta$.
  Let $(\alpha'',\beta'')$ be the first element of $P$ and $s''$ the corresponding slope. 
  Consider the $\CC^\star$-action with a weight $v'' \in \Z^2$ perpendicular to $s''$ such that $v''_1 < 0$. 
  The ideal $I'' := \lim_{t \to 0} t.I$ defines a point in $H^{\Delta,\punc}_\lex(\A^2)$. 
  The element $f''_\alpha$ of its reduced lexicographic Gr\"obner basis arises from $f_\alpha$ by killing all $c^\alpha_\beta$ 
  such that $\beta$ is not on the line through $\alpha$ of slope $s''$. 
  If one of the remaining $c^\alpha_\beta$ didn't vanish, then arguments analogous to those we used for $I'$ 
  would lead to a contradiction to $\supp(I'') = \V(x_1,x_2)$. 
  Therefore, $c^\alpha_\beta = 0$ for all $\beta$ sitting on the line through $\alpha$ of slope $s''$. 
  Taking the next element $(\alpha''',\beta''')$ of $P$ 
  shows that $c^\alpha_\beta = 0$ for all $\beta$ sitting on the line through $\alpha$ of slope $s'''$. 
  Proceeding until the last element of $P$ proves the lemma. 
\end{proof}

The next statement is a refined version of the main theorem of \cite{jpaa}. 

\begin{pro}
\phantomsection\label{pro:C4sumSecondDirection}
  For $i = 0, \ldots, m$, let $I_j \subseteq S$ be ideals in $H^{\Delta_j}_\lex(\A^2)$ supported on distinct horizontal lines $\V(x_2 - \lambda_j)$. 
  Let 
  \[
    I := \bigcap_{j = 0}^m I_j . 
  \]
  Then $\IN_\lex(I) = M_\Delta$, where 
  \[
    \Delta := \cfoursecond_{i = 0}^m \Delta_i. 
  \]
\end{pro}

\begin{proof}
  Consider the Lagrange interpolation polynomial
  \[
    \chi_j := \prod_{j' \in \{0, \ldots, m\} \setminus \{j\}} \frac{x_2 - \lambda_{j'}}{\lambda_j - \lambda_{j'}} \in \CC[x_2] .
  \]
  The sum of all $\chi_j$ is the constant polynomial 1, thus also $(\chi_0 + \ldots + \chi_m)^N = 1$ for all $N$. 
  We expand this power using the multinomial theorem, 
  \begin{equation}
  \phantomsection\label{eqn:multinomial}
    \sum_{N_0 + \ldots + N_m = N} {N \choose N_0, \ldots, N_m} 
    \chi_0^{N_0} \ldots \chi_m^{N_m} = 1 .
  \end{equation}
  
  Let $h_{j'} := h(\Delta_{j'})$, then Lemma \ref{lmm:divisibility} implies that the power $(x_2 - \lambda_{j'})^{h_{j'}}$ lies in $I_{j'}$. 
  We define $H := \max \{ h_0, \ldots, h_m \}$. 
  For all $j$, let
  \[
    L_j := \bigl\{ (N_0, \ldots, N_m): N_0 + \ldots + N_m = N, N_j \geq H \bigr\} . 
  \]
  Then for all $(N_0, \ldots, N_m) \in L_j$, 
  the product $\chi_0^{N_0} \ldots \chi_m^{N_m}$ lies in $\cap_{ j' \in \{0, \ldots, m\} \setminus \{j\} } I_{j'}$. 
  Moreover, as long as we choose $N$ large enough, 
  the union of all $L_j$ is the full indexing set of the sum in \eqref{eqn:multinomial}. 
  However, $L_j$ and $L_{j'}$ may overlap for $j \neq j'$. 
  We therefore pass to a patching of the indexing set by $L'_j \subseteq L_j$ which don't overlap. 
  Then 
  \[
    \epsilon_j := \sum_{(N_0, \ldots, N_m) \in L'_j} {N \choose N_0, \ldots, N_m} \chi_0^{N_0} \ldots \chi_m^{N_m} \in \bigcap_{ j' \in \{0, \ldots, m\} \setminus \{j\} } I_{j'} ,
  \]
  and $\epsilon_0 + \ldots + \epsilon_m = 1$. 
  
  For proving that $\Delta$ is the standard set of $I$, 
  we prove, for each outer corner $\alpha$ of $\Delta$, the existence of an $f_\alpha \in I$
  with initial exponent $\alpha$. 
  Upon defining 
  \[
    \alpha_{2,j} := \left| p_1^{-1}\bigl( p_1(\alpha) \bigr) \cap \Delta_j \right| ,
  \]
  we see that $(\alpha_1,\alpha_{2,j}) \in \N^2 \setminus \Delta_j$, which implies the existence of a polynomial
  $f_j \in I_j$ with initial exponent $(\alpha_1,\alpha_{2,j})$. 
  We write this polynomial as 
  \[
    f_j = x_1^{\alpha_1} p_j + r_j ,
  \]
  where $p_j \in \CC[x_2]$ is monic with initial exponent $\alpha_{2,j}$
  and all terms of $r_j \in S$ are lexicographically smaller than $x_1^{\alpha_1}$. 
  In other words, the powers of $x_1$ appearing in $r_j$ do not exceed $x_1^{\alpha_1 - 1}$. 
  We define 
  \[
    f := \sum_{j \in \{0, \ldots, m\}} f_j \epsilon_j \prod_{j' \in \{0, \ldots, m\} \setminus \{j\}} p_{j'} .
  \]
  Then $f \in I$ since $f_j \in I_j$ and $\epsilon_j \in \cap_{ j' \in \{0, \ldots, m\} \setminus \{j\} } I_{j'}$. 
  Moreover, $f$ expands into
  \[
    f = \sum_{j \in \{0, \ldots, m\}} x_1^{\alpha_1} \epsilon_j \prod_{j' \in \{0, \ldots, m\}} p_{j'}
    + \sum_{j \in \{0, \ldots, m\}} r_j \epsilon_j \prod_{j' \in \{0, \ldots, m\} \setminus \{j\}} p_{j'} .
  \]
  The first sum reduces to $x_1^{\alpha_1} \prod_{j' \in \{0, \ldots, m\}} p_{j'}$, 
  whose initial exponent is $\alpha$, since $\alpha_2 = \alpha_{2,0} + \ldots + \alpha_{2,m}$ by definition of $\alpha_{2,j}$. 
  All terms of the second sum are lexicographically smaller than $x_1^{\alpha_1}$. 
  The polynomial $f$ therefore has the desired properties. 
\end{proof}


\section{Incidences on the \'etale part}
\phantomsection\label{sec:etaleIncidences}

Before proving Theorem \ref{thm:combinatorialDuality} (i), two preliminary remarks are in order. 
The first one is geometric. 
Remember from the Introduction that $H^{n,\et}(\A^2)$ is the open locus of $H^n(\A^2)$ 
on which the Hilbert-Chow morphism $\alpha_n$ is an isomorphism. 
The image of $H^{n,\et}(\A^2)$ under $\alpha_n$ is the symmetric product 
\[
  (\A^2)^{(n)}_\star := \bigl( (\A^2)^n \setminus \Lambda \bigr) / S_n
\]
where $\Lambda$ is the \defn{large diagonal}, i.e., the locus where at least two points in $(\A^2)^n$ coincide, 
and $S_n$ is the symmetric group acting by permutation of the factors \cite[Section 2]{bertin}. 
The identification of points in $H^{n,\et}(\A^2)$ and $(\A^2)^{(n)}_\star$ sends an ideal $I$ to the set $\supp(S/I)$, whose cardinality is $n$. 
We therefore always think of points in $H^{n,\et}(\A^2)$ as collections of $n$ reduced points in $\A^2$. 

The second remark is combinatorial. 
We reformulate the partial ordering $\leq_\et$ in terms of C4 sums. 
A standard set $\Delta$ is the C4 sum in the second direction of its rows, 
\[
  \Delta = \cfoursecond_{\Delta_i \in \rows(\Delta)} \Delta_i ,
\]
where $\rows(\Delta)$ denotes the multiset of rows of $\Delta$. 
Definition \ref{dfn:twoOrderings} then says that $\Delta \leq_\et \Delta'$ if, and only if, 
there exists a partition $\rows(\Delta) = R_0 \coprod \ldots \coprod R_l$ such that 
\[
  \Delta' = \cfoursecond_{ j \in \{0, \ldots, l\} } \Bigl( \cfourfirst_{\Delta_i \in R_i} \Delta_i \Bigr) . 
\]
The rows of $\Delta'$ appear in the parentheses in the last displayed equation; they are just the horizontal concatenation of certain rows of $\Delta$. 
We therefore also express the inequality $\Delta \leq_\et \Delta'$ by saying that \defn{$\Delta'$ arises from $\Delta$ by merging rows}. 

\begin{proof}[Proof of Theorem \ref{thm:combinatorialDuality} (i)]
  Let $\Delta \in \st_n$ and $\rows(\Delta) := \{ \Delta_0, \ldots, \Delta_l \}$ be the multiset of its rows. 
  Let $I \subseteq S$ be an ideal defining a point in $H^{n,\et}(\A^2)$. 
  We write $A := \supp(S/I) \subseteq \A^2$ for the set of $n$ points corresponding to $I$. 
  Propositions \ref{pro:C4sumFirstDirection} and \ref{pro:C4sumSecondDirection} imply that $I$ defines a point in $H^\Delta_\lex(\A^2)$ if, and only if, 
  \begin{itemize}
    \item $|\Delta_0|$ points of $A$ sit on a horizontal line $\V(x_2 - \lambda_0)$, 
    \item $|\Delta_1|$ points of $A$ sit on a horizontal line $\V(x_2 - \lambda_1)$, 
    \item etc., and 
    \item $|\Delta_l|$ points of $A$ sit on a horizontal line $\V(x_2 - \lambda_l)$, 
  \end{itemize}
  and $\lambda_i \neq \lambda_j$ for $i \neq j$. Figure \ref{fig:points} shows an example of a configuration of points defining an ideal $I \in H^\Delta_\lex(\A^2)$. 
  The full moduli space $H^\Delta_\lex(\A^2)$ is the space of configurations of $n$ points in which 
  the horizontal lines may move freely along the $x_2$-axis as long as they don't collide, 
  and the points on each horizontal line $\V(x_2 - \lambda_j)$ may move freely along this line as long as they don't collide. 
  
  \begin{center}
  \begin{figure}[ht]
    \begin{picture}(300,160)
      \put(10,25){\footnotesize $\Delta =$}
      \multiput(40,15)(0,10){3}{\line(1,0){50}}
      \put(40,45){\line(1,0){20}}
      \put(40,55){\line(1,0){10}}
      \multiput(40,15)(10,0){2}{\line(0,1){40}}
      \put(60,15){\line(0,1){30}}
      \multiput(70,15)(10,0){3}{\line(0,1){20}}
      \put(140,15){\line(1,0){150}}
      \put(145,10){\line(0,1){140}}
      \put(140,35){\line(1,0){150}}
      \put(140,70){\line(1,0){150}}
      \put(140,85){\line(1,0){150}}
      \put(140,130){\line(1,0){150}}
      \put(127,32){\footnotesize $\lambda_0$}
      \put(127,67){\footnotesize $\lambda_1$}
      \put(127,82){\footnotesize $\lambda_2$}
      \put(127,127){\footnotesize $\lambda_3$}
      \multiput(175,35)(14,0){2}{\circle*{2}}
      \multiput(200,35)(23,0){2}{\circle*{2}}
      \put(280,35){\circle*{2}}
      \multiput(190,70)(65,0){2}{\circle*{2}}
      \put(240,85){\circle*{2}}
      \multiput(190,130)(15,0){2}{\circle*{2}}
      \multiput(230,130)(23,0){2}{\circle*{2}}
      \put(214,130){\circle*{2}}
  \end{picture}
  \caption{A configuration of points defining an ideal $I \in H^\Delta_\lex(\A^2)$}
  \phantomsection\label{fig:points}
  \end{figure}
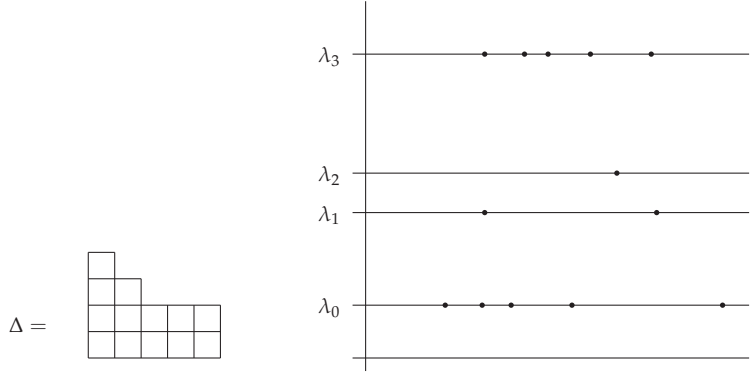
  \end{center}
  
  Within the moduli space $H^n(\A^2)$, however, points have more freedom to move. 
  Consider a configuration $A$ of $n$ reduced points defining a point in $H^{\Delta,\et}(\A^2)$ 
  such the one from Figure \ref{fig:points}. 
  Understanding to which configuration of $n$ reduced points the given configuration $A$ may degenerate 
  \emph{without leaving $H^{\Delta,\et}(\A^2)$} amounts to determining the closure of $H^\Delta_\lex(\A^2)$ in $H^{n,\et}(\A^2)$. 
  Here is how to determine that degeneration: 
  First off, none of the points sitting on a horizontal line $\V(x_2 - \lambda_j)$ may leave the line they sit on. 
  Moreover, none of the points may collide with each other. 
  However, it may and will happen that different horizontal lines merge into one. 
  Of course if this happens, the points sitting on the respective lines will still stay away from each other. 
  The combinatorial shadow of the merging process of horizontal lines in $\A^2$ 
  is encoded in the partial ordering $\leq_\et$, as described earlier in the present section. 
  What we have just proven is the inclusion 
  \[
    \overline{H^\Delta(\A^2)} \subseteq \coprod_{\Delta \leq_\et \Delta'} H^{\Delta'}_\lex(\A^2) .
  \]
  As for the converse inclusion, it's easy to see that each configuration $A'$ of points corresponding to a point in $H^{\Delta'}_\lex(\A^2)$ 
  has arbitrarily close approximations by configurations $A$ of points corresponding to points in $H^\Delta_\lex(\A^2)$. 
\end{proof}


\section{Incidences on the punctual part}
\phantomsection\label{sec:punctualIncidences}

We also start this section, in which we prove Theorem \ref{thm:combinatorialDuality} (ii), with two preliminary remarks. 
The first is once more geometric. 
Remember from the Introduction that Ellingsrud and Str\o mme defined three actions of $\CC^\star$ on $S$ which 
induce Bia\l ynicki-Birula decompositions of three types of Hilbert schemes. 
We have already discussed the first and the third type of action. 
The second type uses a weight $v \in \Z^2$ such that $v_1 < 0 < v_2$. 
The fixed points of this action are again monomial ideals $M_\Delta$ for $\Delta \in \st_n$. 
The limit $\lim_{t \to 0} t.I$ of an ideal defining a point in $H^n(\A^2)$ exists in $S$ only if $I$ is supported on the axis $\V(x_2)$. 
It therefore only makes sense to consider this action on 
\[
  H^{n,\lin}(\A^2) := \bigl\{ \text{ideals } I \subseteq S : \dim(S / I) = n, \, \supp(S / I) = \V(x_2) \bigr\}
\]
The Bia\l ynicki-Birula decomposition then reads 
\[
  H^{n,\lin}(\A^2) = \coprod_{\Delta \in \st_n} H^{\Delta,\lin}_\lex(\A^2) ,
\]
where 
\[
  H^{\Delta,\lin}_\lex(\A^2) := H^\Delta_\lex(\A^2) \cap H^{n,\lin}_\lex(\A^2) .
\]
Note that we have implicitly been using schemes like $H^{\Delta,\lin}_\lex(\A^2)$ 
in Proposition \ref{pro:C4sumFirstDirection} and Lemma \ref{lmm:divisibility} above. 
Ellingsrud and Str\o mme, with a later correction by Huibregtse, 
determined the dimensions of Bia\l ynicki-Birula cells $H^\Delta(\A^2)$, $H^{\Delta,\lin}(\A^2)$ and $H^{\Delta,\punc}(\A^2)$, 
and also related the three types of cells to each other. 
Conca and Valla proved the same result using a different approach. 

\begin{thm}\cite{esBetti, esCells, huibregtseEllingsrud, hilbertBurchMatrices}
\label{thm:affineCells}
  Let $p_1, p_2 : \N^2 \to \N$ be the two projections, $w(\Delta) := |p_1(\Delta)|$ the \defn{width} 
  and $h(\Delta) := |p_2(\Delta)|$ the \defn{height} of $\Delta \in \st_n$. 
  Then 
  \begin{equation*}
    \begin{split}
      H^{\Delta}_\lex(\A^2) & \cong \A^{\left|\Delta\right| + h(\Delta)} , \\
      H^{\Delta,\lin}_\lex(\A^2) & \cong \A^{\left|\Delta\right|} , \\
      H^{\Delta,\punc}_\lex(\A^2) & \cong \A^{\left|\Delta\right| - w(\Delta)} .
    \end{split}
  \end{equation*}
  Moreover, the coefficients of the reduced lexicographic Gr\"obner basis of the universal family over $H^{\Delta}_\lex(\A^2)$, 
  which take the shape
  \[
    f_\alpha = \sum_{\substack{\beta \in \Delta \\ \beta < \alpha}} c^\alpha_\beta x^\beta ,
  \]
  for all outer corners $\alpha$ of $\Delta$, are polynomials in the coordinates of $\A^{\left|\Delta\right| + h(\Delta)}$, 
  and the two closed immersions 
  \[
    H^{\Delta,\punc}_\lex(\A^2) \subseteq H^{\Delta,\lin}_\lex(\A^2) \subseteq H^{\Delta}_\lex(\A^2)
  \]
  are given by passing to certain coordinate subspaces in $\A^{\left|\Delta\right| + h(\Delta)}$. 
\end{thm}

The second remark is a reformulation of the partial ordering $\leq_\punc$ in terms of C4 sums, 
analogous to the reformulation of $\leq_\et$ in the previous section. 
A standard set $\Delta$ is the C4 sum in the first direction of its columns, 
\[
  \Delta = \cfourfirst_{\Delta_j \in \cols(\Delta)} \Delta_j .
\]
Then $\Delta \leq_\punc \Delta'$ if, and only if, each column $\Delta_j$ vertically breaks into one or more pieces
\[
  \Delta_j = \cfoursecond_{ i \in C(\Delta_j) } \Delta'_i ,
\]
for some indexing set $C(\Delta_j)$, such that $\Delta'$ is the C4 sum in the first direction of all pieces of all columns, 
\[
  \Delta' = \cfourfirst_{\substack{\Delta_j \in \cols(\Delta), \\ i \in C(\Delta_j)}} \Delta'_i . 
\]
We therefore also express the inequality $\Delta \leq_\punc \Delta'$ by saying that \defn{$\Delta'$ arises from $\Delta$ by breaking rows apart}. 

\begin{proof}[Proof of Theorem \ref{thm:combinatorialDuality} (ii)]
  Throughout the proof, the standard set $\Delta$ will be fixed. 
  We will consider various $\Delta'$ such that $\Delta <_\punc \Delta'$, 
  always denoting the respective multisets of columns by $\cols(\Delta) = \{\Delta_0, \ldots, \Delta_{w-1}\}$ and 
  $\cols(\Delta') = \{\Delta'_0, \ldots, \Delta'_{w'-1}\}$. 
  We first take a specific $\Delta'$, namely, 
  one for which only one column $\Delta_{j_0}$ of $\Delta$ splits into no more than two columns 
  $\Delta'_{j'_0}$ and $\Delta'_{j'_1}$ of $\Delta'$ and all 
  other columns remain unchanged, so 
  \[
    \cols(\Delta) \setminus \bigl\{ \Delta_{j_0} \bigr\} = \cols(\Delta') \setminus \bigl\{ \Delta'_{j'_0}, \Delta'_{j'_1} \bigr\} .
  \]
  In particular, $w' = w+1$. 
  Let $I' \subseteq S$ be a point in $H^{\Delta',\punc}_\lex(\A^2)$. 
  We will show that $I'$ lies in the closure of $H^{\Delta,\punc}_\lex(\A^2)$. 

  Theorem \ref{thm:affineCells} implies that $H^{\Delta',\lin}_\lex(\A^2)$ is an affine space of dimension $n$ 
  and contains as an open and dense subscheme the moduli space $X^{\Delta'}$ of ideals splitting into 
  \[
    I'' = \bigcap_{j = 0}^{w'-1} I''_j ,
  \]
  where each $I''_j$ lies in $H^{\Delta'_j,\lin}_\lex(\A^2)$ and is supported at $\V(x_1 + c_{j,0},x_2)$ such that $c_{j,0} \neq c_{k,0}$ for $j \neq k$.
  Thus $I''_j$ is given by its reduced lexicographic Gr\"obner basis,
  \[
    I''_j = \langle x_1 + \sum_{b \in \Delta'_j} c_{j,b} x_2^b, x_2^{|\Delta'_j|} \rangle , 
  \]
  so $\supp(I''_j) = \V(x_1 + c_{j,0})$, such that $c_{j,0} \neq c_{k,0}$ for $j \neq k$. 
  In other words, $\Spec(S/I'')$ splits into $w'$ ``tall'' rather than ``fat'' points, 
  the tallness of the factor $\Spec(S/I''_j)$ being the column $\Delta'_j$ of $\Delta'$
  (The same open and dense subscheme of $H^{\Delta',\lin}_\lex(\A^2)$ is used in Section 7.2 of \cite{nakajima}.)
  The middle picture in Figure \ref{fig:tallPoints} shows $\V(I'')$ for an ideal $I''$ 
  in the dense and open subscheme $X^{\Delta'}$ of $H^{\Delta',\lin}_\lex(\A^2)$. 
  
  \begin{center}
  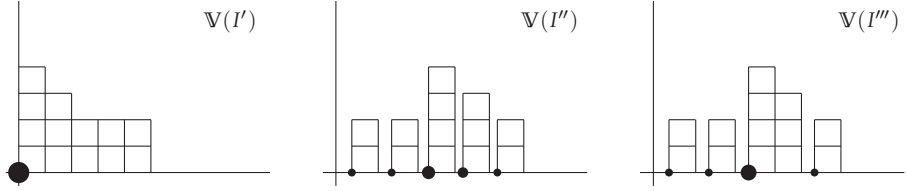
\begin{figure}[ht]
    \begin{picture}(360,90)
      \put(10,15){\line(1,0){100}}
      \put(15,10){\line(0,1){70}}
      \put(85,70){\footnotesize $\V(I')$}
      \multiput(15,25)(0,10){2}{\line(1,0){50}}
      \put(15,45){\line(1,0){20}}
      \put(15,55){\line(1,0){10}}
      \put(25,15){\line(0,1){40}}
      \put(35,15){\line(0,1){30}}
      \multiput(45,15)(10,0){3}{\line(0,1){20}}
      \put(15,15){\circle*{8}}
      \put(130,15){\line(1,0){100}}
      \put(135,10){\line(0,1){70}}
      \put(205,70){\footnotesize $\V(I'')$}
      \put(138,15){
        \multiput(3,0)(10,0){2}{\line(0,1){20}}
        \multiput(3,10)(0,10){2}{\line(1,0){10}}
        \multiput(18,0)(10,0){2}{\line(0,1){20}}
        \multiput(18,10)(0,10){2}{\line(1,0){10}}
        \multiput(32,0)(10,0){2}{\line(0,1){40}}
        \multiput(32,10)(0,10){4}{\line(1,0){10}}
        \multiput(45,0)(10,0){2}{\line(0,1){30}}
        \multiput(45,10)(0,10){3}{\line(1,0){10}}
        \multiput(58,0)(10,0){2}{\line(0,1){20}}
        \multiput(58,10)(0,10){2}{\line(1,0){10}}
        \multiput(3,0)(15,0){2}{\circle*{3}}
        \put(32,0){\circle*{5}}
        \put(45,0){\circle*{4}}
        \put(58,0){\circle*{3}}
      }
      \put(250,15){\line(1,0){100}}
      \put(255,10){\line(0,1){70}}
      \put(325,70){\footnotesize $\V(I''')$}
      \put(258,15){
        \multiput(3,0)(10,0){2}{\line(0,1){20}}
        \multiput(3,10)(0,10){2}{\line(1,0){10}}
        \multiput(18,0)(10,0){2}{\line(0,1){20}}
        \multiput(18,10)(0,10){2}{\line(1,0){10}}
        \multiput(33,0)(10,0){2}{\line(0,1){40}}
        \multiput(33,10)(0,10){3}{\line(1,0){20}}
        \put(33,40){\line(1,0){10}}
        \put(53,0){\line(0,1){30}}
        \multiput(58,0)(10,0){2}{\line(0,1){20}}
        \multiput(58,10)(0,10){2}{\line(1,0){10}}
        \multiput(3,0)(15,0){2}{\circle*{3}}
        \put(33,0){\circle*{6}}
        \put(58,0){\circle*{3}}
      }
  \end{picture}
  \caption{A point in $H^{n,\punc}(\A^2)$ and two approximations in $H^{n,\lin}(\A^2)$}
  \phantomsection\label{fig:tallPoints}
  \end{figure}
  \end{center}

  The ideal $I'$ defines a point in the closed subscheme $H^{\Delta',\punc}_\lex(\A^2)$ of $H^{\Delta',\lin}_\lex(\A^2)$. 
  A point $I''$ in the open and dense subscheme $X^{\Delta'}$ of $H^{\Delta',\lin}_\lex(\A^2)$ 
  can therefore be chosen arbitrarily close to $I'$. 
  The two factors $I''_{j_0}$ and $I''_{j_1}$ of $I''$, 
  whose corresponding schemes are tall points with standard sets $\Delta'_{j_0}$ and $\Delta'_{j_1}$, respectively, are of particular interest to us. 
  The reduced lexicographic Gr\"obner basis of $I''$ contains one polynomial $f_{(w',0)}$ with initial exponent $(w',0)$. 
  The zeros of $f_{(w',0)}(x_1,0)$ are the $w'$ distinct points in $\supp(I'') \subseteq \V(x_2)$. 
  Those coordinates of $H^{\Delta',\lin}_\lex(\A^2) \cong \A^n$ that show up in the polynomial terms for the coefficients of $f_{(w',0)}(x_1,0)$ 
  can be adjusted such that the factors $I''_{j_0}$ and $I''_{j_1}$ of $I''$ collapse into an ideal $I'''_{j_0}$ supported in one point on the $x_1$-axis, 
  whereas all other factors $I''_j$ get transformed into factors $I'''_j$ which are separated from each other and from $I'''_{j_0}$.
  This can be done such that the lexicographic Gr\"obner deformation of ideal $I'''_{j_0}$ has the standard set $\Delta'_{j_0} +_1 \Delta'_{j_1}$. 
  The ideal 
  \[
    I''' := I'''_{j_0} \cap \Bigl( \bigcap_{j \neq j_0} I'''_j \Bigr). 
  \]
  is a deformation of $I''$ which can be chosen arbitrarily close to $I''$. 
  The schemes defined by ideals $I'$, $I''$ and $I'''$ are illustrated in Figure \ref{fig:tallPoints}. 
  What the picture doesn't appropriately indicate is the location of the points in $\supp(I'')$ and $\supp(I''')$. 
  These points should be thought of as sitting arbitrarily close to the origin. 
  
  Let $n_0 := |\Delta_{j_0}| = |\Delta'_{j_0}| + |\Delta'_{j_1}|$, and let $z_0 \in \V(x_2)$ be the support of $I'''_{j_0}$. 
  The $z_0$-translated version of the punctual Hilbert scheme is 
  \[
    H^{n_0,\punc}_{z_0}(\A^2) := \bigl\{ \text{ideals } I \subseteq S : \dim(S / I) = n_0, \, \supp(S / I) = z_0 \bigr\}
  \]
  and for each standard set $\Delta''$ of cardinality $n_0$, the $z_0$-translated version of the corresponding Bia\l ynicki-Birula cell is 
  \[
    H^{\Delta_0,\punc}_{\lex,z_0}(\A^2) := H^{\Delta_0,\lin}_\lex(\A^2) \cap H^{n_0,\punc}_{z_0}(\A^2) .
  \]
  The $z_0$-translated version of Lemma \ref{lmm:simplestPunctualClosure} below implies that 
  the ideal $I'''_{j_0}$ lies in the closure of $H^{\Delta_{j_0},\punc}_{\lex, z_0}(\A^2)$ in $H^{n_0,\punc}_{z_0}(\A^2)$. 
  So there exists an ideal $I^{(4)}_{j_0} \in H^{\Delta_{j_0},\punc}_{\lex, z_0}(\A^2)$ arbitrarily close to the ideal $I'''_{j_0}$. 
  We use this ideal for defining 
  \[
    I^{(4)} := I^{(4)}_{j_0} \cap \Bigl( \bigcap_{j \neq j_0} I'''_j \Bigr) .
  \]
  Then $I^{(4)}$ lies in $H^{\Delta,\lin}_\lex(\A^2)$ by Lemma \ref{pro:C4sumFirstDirection}. 
  
  We now have all tools for showing that $I'$ lies in the closure of $H^{\Delta,\punc}_\lex(\A^2)$. 
  Take an open neighborhood $U \subseteq H^{n,\punc}(\A^2)$ of $I'$. 
  We can find an open $W \subseteq H^{n,\lin}(\A^2)$ such that $U = W \cap H^{n,\punc}(\A^2)$. 
  By choosing $I''$ close enough to $I'$, $I'''$ close enough to $I''$, 
  and $I^{(4)}$ close enough to $I'''$, we can have all these ideals lie in $W$. 
  The ideal $I^{(4)}$ is not supported in the origin, but rather in $w$ points on the $x_1$-axis. 
  Its reduced lexicographic Gr\"obner basis contains one polynomial $f_{(w,0)}$ with initial exponent $(w,0)$. 
  The zeros of $f_{(w,0)}(x_1,0)$ are the $w$ points in $\supp(S/I^{(4)})$. 
  By suitably adjusting some of the coordinates of $H^{\Delta',\lin}_\lex(\A^2) \cong \A^n$ 
  that show up in the polynomial terms for the coefficients of $f_{(w,0)}(x_1,0)$, 
  we kill all trailing terms of $f_{(w,0)}(x_1,0)$. 
  We perform the same adjustment to all elements of the lexicographic Gr\"obner basis of $I^{(4)}$ which, 
  as we remember from Theorem \ref{thm:affineCells}, are polynomials in the coordinates of $H^{\Delta',\lin}_\lex(\A^2) \cong \A^n$. 
  The ideal $I$ generated by these polynomials still lies in $H^{\Delta',\lin}_\lex(\A^2)$, 
  and is supported in the origin by construction. 
  If $I^{(4)}$ is close enough to $I'$, then the $w$ zeros of $f_{(w,0)}(x_1,0)$ are arbitrarily close to the origin. 
  The above-defined readjustment of the coefficients of $f_{(w,0)}(x_1,0)$ therefore leaves the coordinates of $H^{\Delta',\lin}_\lex(\A^2) \cong \A^n$
  arbitrarily close to their original values. 
  The ideal $I$ then still lies in the open neighborhood $W$ of $I'$. 
  It therefore lies in $U = W \cap H^{n,\punc}(\A^2)$. 
  
  We have thus proved the inclusion 
  \[
    H^{\Delta',\punc}_\lex(\A^2) \subseteq \overline{H^{\Delta,\punc}_\lex(\A^2)} 
  \]
  for the specific $\Delta'$ defined above. 
  The same inclusion follows for all $\Delta'$ such that $\Delta <_\punc \Delta'$ by induction over the poset $\st_n$. 
  This establishes the first half of the proof. 
  
  As for the second half, we first consider any standard set $\Delta'$ of cardinality $n$, 
  and introduce a method for finding out whether or not $\Delta <_\punc \Delta'$. 
  We iteratively define multisets 
  \begin{itemize}
    \item $C_0$ (``to be checked if splittable''), 
    \item $C_1$ (``splits''), 
    \item $C_2$ (``isn't splittable''), 
    \item $C'_0$ (``to be checked if arises as a split product''), 
    \item $C'_1$ (``arises as a split product''), and 
    \item $C'_2$ (``doesn't arise as a split product'') 
  \end{itemize}
  by the algorithm from Figure \ref{figure:alg}. 
  
  \begin{figure}[h]
  \begin{algorithmic}
    \State $C_0 \gets \cols(\Delta)$, $C_1 \gets \emptyset$, $C_2 \gets \emptyset$
    \State $C'_0 \gets \cols(\Delta')$, $C'_1 \gets \emptyset$, $C'_2 \gets \emptyset$
    \State $G \gets \text{ the identical function } C_0 \to \cols(\Delta)$
      \While{$C_2 = \emptyset$}
        \If{some element of $C_0$ is taller than the shortest element of $C'_0$}
          \State $c \gets$ any element from $C_0$ that is taller than the shortest element of $C'_0$
          \State split $c$ into $c = d +_2 e$, where $d \in C'_0$ and $e$ is possibly $\emptyset$
          \State $C_0 \gets \left( C_0 \setminus \{c\} \right) \cup \{e\}$
          \State $G \gets \text{ the function } C_0 \to \cols(\Delta)$ sending all $f \neq c$ to $G(f)$ and $e$ to $G(c)$
          \State $C_1 \gets \bigl\{ \cfourfirst_{f \in G^{-1}(g)} f : g \in \cols(\Delta) \bigr\}$
          \State $C'_0 \gets C'_0 \setminus \{d\}$
          \State $C'_1 \gets C'_1 \cup \{d\}$
        \Else
          \State $C_2 \gets C_0$
          \State $C'_2 \gets \cols(\Delta') \setminus C'_1$
        \EndIf
      \EndWhile
    \State return $(C_1,C_2,C'_1,C'_2)$
  \end{algorithmic}
  \caption{The iterative non-deterministic definition of $C_1$ and $C'_1$}
  \phantomsection\label{figure:alg}
  \end{figure}

  This iteration will eventually terminate, but it isn't deterministic. 
  Its output depends on the choice of $c$ and $d$. 
  It's not hard to see that the space of all outputs of the algorithm includes the quadruple $(\cols(\Delta),\emptyset,\cols(\Delta'),\emptyset)$ if, 
  and only if, $\Delta \leq_\punc \Delta'$. 
  From now on we deal with the complementary case, $\Delta \nleq_\punc \Delta'$, 
  so $C_2,C'_2 \neq \emptyset$ for every output $(C_1,C_2,C'_1,C'_2)$ of the algorithm. 
  We define, for such every such output, 
  \begin{equation*}
    \begin{split}
      \Delta_j & := \cfourfirst_{c \in C_j} c , \\ 
      \Delta'_j & := \cfourfirst_{c \in C'_j} c 
    \end{split}
  \end{equation*}
  for $j = 1,2$. The definitions imply that 
  \begin{itemize}
    \item $|\Delta_1| = |\Delta'_1|$, 
    \item $n_2 := |\Delta_2| = |\Delta'_2| \neq 0$, and 
    \item every column of $\Delta_2$ is shorter than every column of $\Delta'_2$. 
  \end{itemize}

  Let $I'$ be a point in $H^{\Delta',\punc}_\lex(\A^2)$. 
  For finishing the second half of the proof, 
  we have to show that $I'$ is not contained in the closure of $H^{\Delta,\punc}_\lex(\A^2)$. 
  By wiggling certain parameters appearing in the generators of $I'$, 
  similarly as we did in the first half of the proof, 
  we deform $I'$ into an ideal $I'' \in H^{\Delta',\lin}_\lex(\A^2)$ which splits into 
  \begin{equation*}
    \begin{split}
      I'' & = I''_1 \cap I''_2 , \\
      I''_1 & := \bigcap_{c \in C'_1} I''_c ,
    \end{split}
  \end{equation*}
  where 
  \begin{itemize}
    \item each ideal $I''_c$ is supported in a point $z'_c$ on the $x_1$-axis, 
    \item $I''_2$ is supported in $0$, 
    \item all $z'_c$ are distinct, and distinct from $0$, 
    \item each $I''_c$ has the standard set $\Delta'_c$, and 
    \item $I''_2$ has the standard set $\Delta'_2$. 
  \end{itemize}
  By construction of $\Delta'_1$, and by the same arguments that we employed in the first half of the proof, 
  the ideal $I''_1$ is arbitrarily close to the closure of $H^{\Delta_1,\punc}_\lex(\A^2)$. 
  Moreover, each ideal $J \in H^{\Delta,\punc}_\lex(\A^2)$ can analogously be deformed into 
  \begin{equation*}
    \begin{split}
      J' & = J'_1 \cap J'_2 , \\
      J'_1 & := \bigcap_{c \in C_1} J'_c ,
    \end{split}
  \end{equation*}
  such that $J'_1$ is arbitrarily close to $H^{\Delta_1,\punc}_\lex(\A^2)$. 
  It all therefore boils down to showing that $I''_2$ is not contained in the closure of $H^{\Delta_2,\punc}_\lex(\A^2)$. 
  
  Consider the locus 
  \[
    U := H^{n_2,\punc}(\A^2) \setminus \bigl\{ I \subseteq S : \rank \bigl( (S / I) \cap \langle x_2^0, \ldots, x_2^{h(\Delta_2)} \rangle \bigr) 
    \leq h(\Delta_2) \bigr\} .
  \]
  This is an open subscheme of $H^{n_2,\punc}(\A^2)$, since the locus we remove is defined by rank inequalities, 
  i.e., the closed subscheme defined by a number of determinants. 
  Then, on the one hand, $I_2''$ lies in $U$, since $S / I_2''$ has the basis $(x^\alpha)_{\alpha \in \Delta'_2}$, 
  and the $h(\Delta_2)+1$ monomials $x_2^0, \ldots, x_2^{h(\Delta_2)}$ are a subfamily of that basis. 
  On the other hand, Lemma \ref{lmm:divisibility} implies that the lexicographic Gr\"obner basis of an ideal $I_2 \in H^{\Delta_2,\punc}_\lex(\A^2)$
  contains the polynomial 
  \[
    f_{(0,h(\Delta_2))} := x_2^{h(\Delta_2)} . 
  \]
  The family $(x_2^0, \ldots, x_2^{h(\Delta_2)})$ is therefore linearly dependent in $S / I_2$. 
  This shows that $U \cap H^{\Delta_2,\punc}_\lex(\A^2) = \emptyset$. 
  We have thus found an open neighborhood of $I_2''$ which doesn't meet $H^{\Delta_2,\punc}_\lex(\A^2)$. 
\end{proof}

\begin{lmm}
\phantomsection\label{lmm:simplestPunctualClosure}
  Consider the two-dimensional standard set
  \[
    \Delta := \{0\} \times [0,n-1] ,
  \]
  and any two-dimensional standard set $\Delta'$ of the same cardinality, $n$. 
  Then $\Delta \leq_\punc \Delta'$, 
  and $H^{\Delta',\punc}_\lex(\A^2)$ is contained in the closure of $H^{\Delta,\punc}_\lex(\A^2)$ in $H^{n,\punc}(\A^2)$. 
\end{lmm}

\begin{proof}
  The first assertion is immediate from the definition of $\leq_\punc$. 
  As for the second assertion, $H^{n,\punc}(\A^2)$ is irreducible of dimension $n-1$ \cite{briancon}. 
  Moreover, 
  \[
    H^\Delta_\lex(\A^2) \cap H^{n,\punc}(\A^2) \cong \A^{n-1} ,
  \]
  the isomorphism sending a point with coordinates $c_1,\ldots,c_{n-1}$ on the right-hand side to the ideal 
  \[
    I := \langle x_1 + c_{n-1}x_2^{n-1} + \ldots + c_1x_2, x_2^{n-1} \rangle
  \]
  defining a point on the left-hand side. 
  The closure of $H^{\Delta,\punc}_\lex(\A^2)$ is therefore the whole $H^{n,\punc}_\lex(\A^2)$. 
\end{proof}



\section{Weak incidences}
\phantomsection\label{sec:generalIncidences}

A decomposition \eqref{eqn:generalStratification} of a topological space $X$ into locally closed subspaces $X_i$ 
raises the problem \defn{incidences} among the subspaces $X_i$. 
We speak of \defn{strong incidence} if $\overline{X_i} \supseteq X_j$ 
and of \defn{weak incidence} if $\overline{X_i} \cap X_j \neq \emptyset$. 
If \eqref{eqn:generalStratification} is a stratification, 
then strong and weak incidence are the same thing, 
and the incidence problem is encoded in the poset $\es$. 

Given that the intersections of decomposition \eqref{bbA2} of $H^n(\A^2)$ 
with subschemes $H^{n,\et}(\A^2)$ and $H^{n,\punc}(\A^2)$ are both stratifications, 
one wonders if that decomposition itself is a stratification. It isn't. 

\begin{nonex}
  Consider the standard sets $\Delta$ and $\Delta'$ from Figure \ref{fig:non-example1}. 
  The scheme $H^{\Delta',\punc}_\lex(\A^2)$ is contained in the closure of $H^{\Delta,\punc}_\lex(\A^2)$ in $H^{6,\punc}(\A^2)$, 
  and therefore, $H^{\Delta'}_\lex(\A^2)$ intersects the closure of $H^\Delta_\lex(\A^2)$ in $H^6(\A^2)$ nontrivially. 
  However, $H^{\Delta',\et}_\lex(\A^2)$ doesn't intersect the closure of $H^{\Delta,\et}_\lex(\A^2)$ in $H^{6,\et}(\A^2)$. 
  The decomposition of $H^6(\A^2)$ from \eqref{bbA2} is therefore not a stratification. 
\end{nonex}

\begin{center}
\begin{figure}[ht]
  \begin{picture}(170,50)
    \put(10,20){\footnotesize $\Delta =$}
    \multiput(40,10)(0,10){2}{\line(1,0){30}}
    \put(40,30){\line(1,0){20}}
    \put(40,40){\line(1,0){10}}
    \multiput(40,10)(10,0){2}{\line(0,1){30}}
    \put(60,10){\line(0,1){20}}
    \put(70,10){\line(0,1){10}}
    \put(90,20){\footnotesize $\Delta' =$}
    \multiput(120,10)(0,10){2}{\line(1,0){40}}
    \multiput(120,30)(0,10){2}{\line(1,0){10}}
    \multiput(120,10)(10,0){2}{\line(0,1){30}}
    \multiput(140,10)(10,0){3}{\line(0,1){10}}
  \end{picture}
  \caption{Two standard sets showing that $H^n(\A^2)$ is not stratified by lexicographic Gr\"obner basins}
  \phantomsection\label{fig:non-example1}
\end{figure}
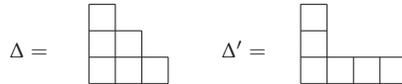
\end{center}

Similar non-examples can be constructed for all $n \geq 4$. 
However, Theorem \ref{thm:combinatorialDuality} and the results from Section \ref{sec:c4twoDirections}
provide ample supply of examples for weak incidence between Gr\"obner basins in $H^n(\A^2)$. 

\begin{ex}
\label{ex:oneAndTwo}
  Consider two standard sets $\Delta$ and $\Delta'$ such that 
  \[
    \Delta = \cfoursecond_{r \in \rows(\Lambda)} \bigl( \cfourfirst_{b \in r} \Delta_b \bigr) 
  \]
  for a collection of standard sets $\Delta_b$ indexed by the boxes $b$ appearing a standard set $\Lambda$, 
  arranged row by row, and analogously, 
  \[
    \Delta' = \cfoursecond_{r' \in \rows(\Lambda')} \bigl( \cfourfirst_{b' \in r'} \Delta_{b'} \bigr) 
  \]
  where $\Lambda \leq_\et \Lambda'$, $f: \Lambda \to \Lambda'$ 
  is a bijection sending elements from one and the same row of $\Lambda$ to elements from one and the same row of $\Lambda'$, and 
  $\Delta_b \leq_\punc \Delta'_{f(b)}$ for all boxes $b$ of $\Lambda$. 
  Then a proof similar to those from Sections \ref{sec:etaleIncidences} and \ref{sec:punctualIncidences} shows that 
  \[
    \overline{H^\Delta(\A^2)} \cap H^{\Delta'}(\A^2) \neq \emptyset .
  \]
  This is illustrated in Figure \ref{fig:weakIncidence}. An ideal $I \in H^\Delta(\A^2)$ splits into $I = \cap_{b \in \Lambda} I_b$, 
  where $\IN_\lex(I_b) = M_{\Delta_b}$, indexed by the boxes $b$ appearing in standard set $\Lambda$ 
  such that ideal indexed by boxes from one and the same row of $\Lambda$ are supported on one and the same horizontal line in $\A^2$.
  Ideals having this property come in a subfamily of $H^\Delta(\A^2)$. 
  When passing to the limit in $H^n(\A^2)$ of such a family, at least three types of degenerations may happen: 
  \begin{itemize}
    \item Horizontal lines may merge. 
    \item Each ideal $I_b$ may degenerate into an ideal $I'_{b'}$ 
    with a standard set $\Delta'_{b'}$ such that $\Delta_b <_\punc \Delta'_{b'}$. 
    \item Punctual ideals lying on the same horizontal line may merge horizontally. 
  \end{itemize}
  All three types of degeneration show up in Figure \ref{fig:weakIncidence}. 
  However, degeneration may also happen ``diagonally''---in a way not captured by the three bulleted items above. 
\end{ex}

  \begin{center}
  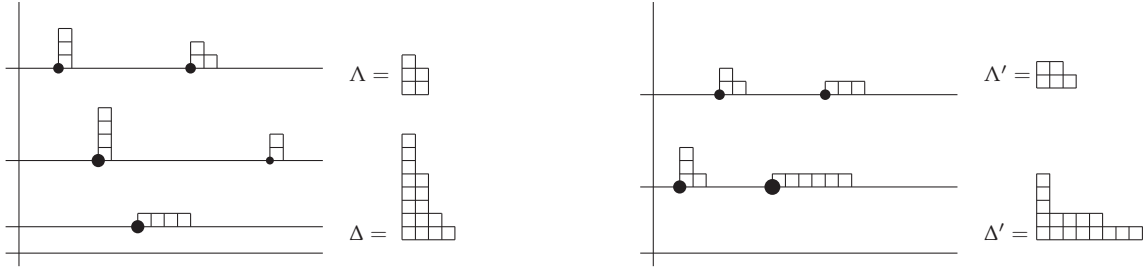
\begin{figure}[ht]
    \begin{picture}(450,120)
      \put(10,15){\line(1,0){120}}
      \put(15,10){\line(0,1){100}}
      \multiput(10,25)(0,25){2}{\line(1,0){120}}
      \put(10,85){\line(1,0){120}}
      \put(30,85){\circle*{4}}
      \multiput(30,85)(5,0){2}{\line(0,1){15}}
      \multiput(30,90)(0,5){3}{\line(1,0){5}}
      \put(80,85){\circle*{4}}
      \multiput(80,85)(5,0){2}{\line(0,1){10}}
      \put(80,95){\line(1,0){5}}
      \put(80,90){\line(1,0){10}}
      \put(90,85){\line(0,1){5}}
      \put(45,50){\circle*{5}}
      \multiput(45,50)(5,0){2}{\line(0,1){20}}
      \multiput(45,55)(0,5){4}{\line(1,0){5}}
      \put(110,50){\circle*{3}}
      \multiput(110,50)(5,0){2}{\line(0,1){10}}
      \multiput(110,55)(0,5){2}{\line(1,0){5}}
      \put(60,25){\circle*{5}}
      \multiput(60,25)(5,0){5}{\line(0,1){5}}
      \put(60,30){\line(1,0){20}}
      \put(140,80){\footnotesize $\Lambda =$}
      \multiput(160,75)(5,0){2}{\line(0,1){15}}
      \put(170,75){\line(0,1){10}}
      \multiput(160,75)(0,5){3}{\line(1,0){10}}
      \put(160,90){\line(1,0){5}}
      \put(140,20){\footnotesize $\Delta =$}
      \multiput(160,20)(5,0){2}{\line(0,1){40}}
      \put(170,20){\line(0,1){25}}
      \put(175,20){\line(0,1){10}}
      \put(180,20){\line(0,1){5}}
      \multiput(160,20)(0,5){2}{\line(1,0){20}}
      \put(160,30){\line(1,0){15}}
      \multiput(160,35)(0,5){3}{\line(1,0){10}}
      \multiput(160,50)(0,5){3}{\line(1,0){5}}
      \put(250,15){\line(1,0){120}}
      \put(255,10){\line(0,1){100}}
      \multiput(250,40)(0,35){2}{\line(1,0){120}}
      \put(280,75){\circle*{4}}
      \multiput(280,75)(5,0){2}{\line(0,1){10}}
      \put(290,75){\line(0,1){5}}
      \put(280,80){\line(1,0){10}}
      \put(280,85){\line(1,0){5}}
      \put(320,75){\circle*{4}}
      \multiput(320,75)(5,0){4}{\line(0,1){5}}
      \put(320,80){\line(1,0){15}}
      \put(265,40){\circle*{5}}
      \multiput(265,40)(5,0){2}{\line(0,1){15}}
      \put(275,40){\line(0,1){5}}
      \put(265,45){\line(1,0){10}}
      \multiput(265,50)(0,5){2}{\line(1,0){5}}
      \put(300,40){\circle*{6}}
      \multiput(300,40)(5,0){7}{\line(0,1){5}}
      \put(300,45){\line(1,0){30}}
      \put(380,80){\footnotesize $\Lambda' =$}
      \multiput(400,77.5)(5,0){3}{\line(0,1){10}}
      \put(415,77.5){\line(0,1){5}}
      \multiput(400,77.5)(0,5){2}{\line(1,0){15}}
      \put(400,87.5){\line(1,0){10}}
      \put(380,20){\footnotesize $\Delta' =$}
      \multiput(400,20)(5,0){2}{\line(0,1){25}}
      \multiput(410,20)(5,0){4}{\line(0,1){10}}
      \multiput(430,20)(5,0){3}{\line(0,1){5}}
      \multiput(400,20)(0,5){2}{\line(1,0){40}}
      \put(400,30){\line(1,0){25}}
      \multiput(400,35)(0,5){3}{\line(1,0){5}}
  \end{picture}
  \caption{A point in $H^{n,\punc}(\A^2)$ and two approximations in $H^{n,\lin}(\A^2)$}
  \phantomsection\label{fig:weakIncidence}
  \end{figure}
  \end{center}

We now present three necessary conditions for weak incidence on $H^n(\A^2)$. 
One of the conditions will use the \defn{natural partial ordering}, or \defn{dominance partial ordering} on $\st_n$, 
which is characterized by the two equivalent conditions \cite[p.7]{Macdonald}
\begin{equation*}
\begin{split}
  \Delta \leq \Delta' : & \iff \forall j : 
  \sum_{\substack{\text{columns } C \text{ in }\Delta \\ \text{ up to } j}}h(C) 
  \geq 
  \sum_{\substack{\text{columns } C' \text{ in }\Delta' \\ \text{ up to } j}}h(C') \\
  & \iff \forall i : 
  \sum_{\substack{\text{rows } R \text{ in }\Delta \\ \text{ up to } i}}h(R) 
  \leq 
  \sum_{\substack{\text{rows } R' \text{ in }\Delta' \\ \text{ up to } i}}h(R') .
\end{split}
\end{equation*}

\begin{pro}
\phantomsection\label{pro:lexByCols}
  If $H^{\Delta'}_\lex(\A^2)$ meets the closure of $H^\Delta_\lex(\A^2)$ in $H^n(\A^2)$, then 
  \begin{enumerate}[(i)]
    \item $\Delta \leq \Delta'$, 
    \item the tuple of rows of $\Delta$ is lexicographically smaller than the tuple of rows of $\Delta'$, and 
    \item the tuple of columns of $\Delta$ is lexicographically larger than the tuple of columns of $\Delta'$. 
  \end{enumerate}
\end{pro}

\begin{proof}
  (i) is in the same spirit as the last paragraph from the proof of Theorem \ref{thm:combinatorialDuality} (ii). 
  For the time being, we use the notation $\N_j := \{ \alpha \in \N^2 : p_1(\alpha) \leq j \}$ and $\Delta_j := \Delta \cap \N^2_j$. 
  Consider the closed subschemes 
  \[
    Y_j := \bigl\{ I \in H^n(\A^2) : \rank \bigl( (S / I) \cap \langle x^\alpha : \alpha \in \N_j \rangle \bigr) 
    \leq \bigl| \Delta_j \bigr| \bigr\} 
  \]
  of $H^n(\A^2)$ and their intersection 
  \[
    Y := \bigcap_{j = 0}^{w(\Delta)} Y_j .
  \]
  For each $I \in H^\Delta_\lex(\A^2)$ and each outer corner $\alpha$ of $\Delta$, 
  there exists a unique polynomial $f_\alpha \in I$ with lex-initial exponent $\alpha$ and trailing exponents in $\Delta$. 
  This implies that $I \in Y$, more precisely, that $H^\Delta_\lex(\A^2)$ is a subscheme of $Y$. 
  (Thorough schematic arguments for statements like this is given in \cite{strata}, but shall not be carried out here.)
  If $I'$ is an ideal living in the intersection under discussion, 
  then $I'$ arises as the limit of a flat family in $H^\Delta_\lex(\A^2)$, so of a flat family in $Y$. 
  The limit $I'$ therefore also lies in $Y$. 
  For all $j$, the family $(x^\alpha)_{\alpha \in \Delta'_j}$ 
  is a basis of the subspace of $(S / I') \cap \langle x^\alpha : \alpha \in \N^2_j \rangle$. 
  Thus 
  \[
    |\Delta'_j| = \rank \bigl( (S / I') \cap \langle x^\alpha : \alpha \in \N^2_j \rangle \bigr) \leq |\Delta_j| ,
  \]
  as was claimed. 
  
  (ii) and (iii) are consequences of (i) by \cite[(1.10), p.7]{Macdonald}. 
\end{proof}

We conclude the paper with three remarks on the conditions from the last proposition. 
\begin{itemize}
  \item Partial orderings $\leq_\et$ and $\leq_\punc$ are both refinements $\leq$ in the sense that 
  $\alpha \leq_\et \beta$ implies $\alpha \leq \beta$, and 
  $\alpha \leq_\punc \beta$ implies $\alpha \leq \beta$. 
  Does there exist a common refinement of $\leq_\et$ and $\leq_\punc$? 
  \item Evain has studied decompositions and weak incidences 
  among locally closed schemes related to the punctual Hilbert scheme \cite{laurentIncidences}. 
  The schemes he studied are not quite $H^{n,\punc}(\A^2)$, 
  but nevertheless, Theorem 8 of the cited paper is related to Proposition \ref{pro:lexByCols}. 
  \item The conditions for weak incidence from our proposition here are easily checked not to be sufficient. 
\end{itemize}


\bibliography{references}
\bibliographystyle{amsalpha}

\end{document}